\documentclass[10pt]{article}
\usepackage{amsmath,amssymb,amsthm, amsfonts}
\usepackage{hyperref}
\usepackage{graphicx}
\usepackage{url}

\usepackage{dsfont} 
\usepackage{color}
\definecolor{red}{rgb}{1,0,0}

\definecolor{blue}{rgb}{0,0,1}

\definecolor{green}{rgb}{0,.6,0}

\usepackage{float}

\usepackage{tikz}

\newtheorem{thm}{Theorem}[section]

\newtheorem{conj}[thm]{Conjecture}

\newtheorem{quest}[thm]{Question}

\theoremstyle{definition}

\theoremstyle{definition}

\theoremstyle{definition}

\theoremstyle{definition}
\newtheorem*{claim}{Claim}

\newcommand{\bit}{\begin{itemize}}
\newcommand{\eit}{\end{itemize}}
\newcommand{\ben}{\begin{enumerate}}
\newcommand{\een}{\end{enumerate}}
\newcommand{\beq}{\begin{equation}}
\newcommand{\eeq}{\end{equation}}
\newcommand{\bea}{\begin{eqnarray*}} 
\newcommand{\eea}{\end{eqnarray*}}
\newcommand{\bpf}{\begin{proof}}
\newcommand{\epf}{\end{proof}\ms}
\newcommand{\bmt}{\begin{bmatrix}}
\newcommand{\emt}{\end{bmatrix}}
\newcommand{\ms}{\medskip}

\title{Rainbow matchings of size $m$ in graphs with total color degree at least $2mn$}
\author{J\"{u}rgen Kritschgau\thanks{Department of Mathematics, Iowa State University, \texttt{jkritsch@iastate.edu}}}

\begin{document}
\maketitle

\begin{abstract}
    The existence of a rainbow matching given a minimum color degree, proper coloring, or triangle-free host graph has been studied extensively. This paper, generalizes these problems to edge colored graphs with given total color degree. In particular, we find that if a graph $G$ has total color degree $2mn$ and satisfies some other properties, then $G$ contains a matching of size $m$; These other properties include $G$ being triangle-free, $C_4$-free, properly colored, or large enough. 
\end{abstract}

\section{Introduction}

Given a graph $G$, let $V(G)$ denote the vertex set of $G$ and $E(G)$ denote the edge set of $G$. If $S\subseteq V$, then $G[S]$ denotes subgraph induced by the vertices in $S$. A graph $G$ is an $m$-matching if $G$ contains exactly $m$ edges, $2m$ vertices, and $e\cap e'=\{\}$ for all edges $e\neq e'$ in $E(G)$. An edge coloring $c:E(G)\to [r]=\{1,\dots, r\}$ is an assignment of colors to edges. A proper edge coloring of a graph is an edge coloring such that $c(e)\neq c(e')$ whenever $e\cap e'\neq \varnothing$ and $e\neq e'$. The colors used on a graph will be denoted $c(G)$, and $R$ will denote a generic color class. If $X,Y\subseteq V(G)$, then $c(X,Y)$ will denote the set of colors used on edges of the form $xy$, where $x\in X$, $y\in Y$. A graph $G$ is rainbow under $c$ if $c$ is injective on $E(G)$. In particular, a rainbow matching is a matching where each edge receives a unique color within the matching. The color degree of a vertex $v$ is denoted $\hat d_G(v)$, which is the number of colors $c$ assigns to edges incident upon $v$ in $G$; when it is clear from the context what $G$ is, we will drop the subscript. Let $\hat d^R(v)$ denote the color $R$ degree of $v$, that is, the number of $R$ colored edges incident upon $v$. The total color degree of $G$ with respect to $c$ is the sum of all the color degrees in the graph and denoted $$\hat d(G)=\sum_{v\in V(G)}\hat d(v).$$ The average color degree of a graph $G$ is obtained by dividing the total color degree by $|V(G)|$, and is an equivalent notion. The minimum color degree of $G$ is denoted $\hat \delta(G)$.   Finally, let $G-v$ denote the graph $G$ with the vertex $v$ deleted, and $G-R$ denote the graph $G$ with the edges in color class $R$ deleted.  When convenient, I will let $c(e)$ denote a color class so that $G-c(e)$ denotes the graph $G$ without the edges in color class containing the edge $e$.

Rainbow matchings in graphs were originally studied due to their connection to transversals of Latin squares \cite{R},\cite{S}.   However, the existence of rainbow matchings has also been studied in its own right. In \cite{LW}, Li and Wang conjectured that any graph with $\hat \delta(G)\geq m\geq 4$ contains a rainbow matching of size $\lceil \frac{m}{2}\rceil$. This conjecture was partially confirmed in \cite{LSWW}, and fully confirmed in \cite{KY}.

 Wang asked for a function $f$ such that any such that any properly edge colored graph $G$ with $|V(G)|\geq f(\hat \delta(G))$ contains a rainbow matching of size $\hat \delta(G)$ \cite{W}. Diemunsch et al. determined that $|V(G)|\geq\frac{98}{23}\hat \delta(G)$ is sufficient \cite{DFLMPW}. This problem was generalized to find a function $f$ such that any edge colored graph $G$ with $|V(G)|\geq f(\hat \delta(G))$ contains a rainbow matching of size $\hat \delta(G)$.  The authors of \cite{KPY} found that $|V(G)|\geq \frac{17}{4}\hat \delta(G)^2$ sufficed. This was improved to $4\hat \delta(G)-4$ for $\hat \delta(G)\geq 4$ in \cite{GS} and \cite{LT} independently.

 Local  Anti-Ramsey theory asks Anti-Ramsey type questions with assumptions about the local structure of the host graph. In particular, Local Anti-Ramsey theory is about the minimum $k$ such that any coloring of $K_n$ with $\hat \delta(G)\geq k$ contains a rainbow copy of $H$. In this vein, Wang's question can be posed as follows: given $k$, what is the smallest $N$ such that any properly edge colored graph $G$ with $|V(G)|\geq N$ and $\hat \delta (G)\geq k$ contains a rainbow matching of size $k$? Furthermore, proper edge-coloring and triangle-free properties play similar roles in restricting the structure of a host graph.
 
 The local assumptions in Anti-Ramsey theory are interesting in so far as they highlight the relationship between a local parameter and the target graph. In much of the rainbow matching literature, there are confounding local assumptions. For example, \cite{DFLMPW}, \cite{LWZ}, and \cite{W} all consider hosts graphs that have a prescribed minimum color degree and are properly edge colored. In this case, an intuitive interpretation is that the minimum color degree and proper edge-coloring properties spread the colors apart in the host graph. As one would expect, this makes it easier to find a large rainbow matching. However, it is unclear whether both the minimum color degree and proper edge coloring property are necessary to find a large matching.
 
 The goal of this paper is to shed light on the relationship between local assumptions and rainbow matchings. Rather than considering host graphs with a prescribed  minimum color degree, we will consider host graphs with a prescribed average color degree. This is motivated in part by a question posed during the Rocky Mountain and Great Plains Graduate Research Workshop in Combinatorics in 2017. 
 
 \begin{quest}\label{avedegreeq}
 If $G$ is an edge colored graph on $n$ vertices with $\hat d(G)\geq 2mn$, does $G$ contain a rainbow matching of size $m$?
 \end{quest}
 
  Section \ref{trianglefree} considers this question for triangle-free and $C_4$-free host graphs. In the case of triangle-free graphs, we will prove the slightly stronger statement that if $G$ is a graph with $\hat d(G)>2mn$, then there exists a rainbow matching of size $m+1$.  Section \ref{proper} pertains to properly edge colored host graphs. Finally, Section \ref{general} considers edge colored graphs with total color degree $2mn$, but with no further assumptions.

\section{Triangle-free and $C_4$-free Graphs}\label{trianglefree}

In this section, we consider triangle-free and $C_4$-free graphs. 

\begin{thm}\label{trifreethm} Let $G$ be a triangle-free graph on $n$ vertices. Let $c$ be an edge coloring of $G$ with $\hat d(G)>2mn$. Then $c$ admits a rainbow matching of size $m+1$.
\end{thm}
\begin{proof}
Let $M$ be a maximum rainbow matching of size $k\leq m$ with edges $u_iv_i$ for $1\leq i\leq k$, such that the number of colors appearing on $G[V(G)\setminus V(M)]=H$ is maximized. Without loss of generality, suppose that $c(u_iv_i)=i$.  Since $G$ is triangle-free, $\hat d(u_i)+\hat d(v_i)\leq n$ for all $u_iv_i\in E(M)$. If $H$ has an edge $e$, then $c(e)\in [k]$. Without loss of generality, suppose that $c(H)=[j]$ for some $0\leq j\leq k$. Then for all $v\in V(H)$, we have $\hat d(v)\leq k+j$. Notice that if there exists an edge $e\in H$ with $c(e)=i$, then we can swap $e$ and $u_iv_i$ to conclude that $\hat d(u_i)+\hat d(v_i)\leq 2(j+k)$. 

Now consider

\begin{align*}
    2mn&<\sum_{i=1}^k\hat d(u_i)+\hat d(v_i)+\sum_{v\in H} \hat d_G(v)\\
    &\leq \sum_{i=1}^j\hat d(u_i)+d(v_i)+\sum_{i=j+1}^k\hat d(u_i)+\hat d(v_i)+\sum_{v\in H} \left( \hat d_H(v)+k\right)\\
    &\leq 2j(k+j)+(k-j)n+(n-2k)(j+k)\\
    &= 2jk+2j^2+2nk-2jk-2k^2\\
    &\leq 2j^2-2k^2+2nk\\
    &\leq 2nm.
\end{align*}

This is a contradiction; therefore, $k\geq m+1$. 
\end{proof}

A key element to the proof of Theorem \ref{trifreethm} is the bound $\hat d(v)+\hat d(u)\leq n$ where $uv$ is an edge in a maximal matching. We can obtain a similar bound in $C_4$-free graphs in order to prove the next theorem.

\begin{thm}\label{c4freethm}
Let $G$ be a $C_4$-free graph on $n$ vertices. Let $c$ be an edge coloring of $G$ with $\hat d(G)\geq 2mn$. Then $c$ admits a rainbow matching of size $m$.
\end{thm}

\begin{proof}
Let $M$ be a maximum rainbow matching of size $k< m$ with edges $u_iv_i$ for $1\leq i\leq k$, such that the number of colors appearing on $G[V(G)\setminus V(M)]=H$ is maximized. Without loss of generality, suppose that $c(u_iv_i)=i$. Since $G$ is $C_4$-free, $\hat d(u_i)+\hat d(v_i)\leq n+1$ for all $u_iv_i\in E(M)$. If $H$ has an edge $e$, then $c(e)\in [k]$. Without loss of generality, suppose that $c(H)=[j]$ for $0\leq j\leq k$. 

\begin{claim} If $xy\in E(H)$ with $c(xy)=i\leq j$, then $\hat d(u_i)+\hat d (v_i)\leq 2j+2k$.\end{claim}

Notice that $x,y$ each see at most $j$ colors in $H$. Since $xy$ can share at most two edges with any edge in $M$ without creating a $C_4$ subgraph, we have $|c(\{u_i,v_i\},e)|\leq k$. Thus, $\hat d(x)+\hat d(y)\leq 2j+2k$. By swapping $u_iv_i$ and $e$, we obtain the desired bound on $\hat d(u_i)+\hat d(v_i)$.

 Furthermore, $\sum_{v\in H} \hat d_G(v)\leq (n-2k)(j+k)+k$. The $(n-2k)j$ term comes from the fact that $H$ has $n-2k$ vertices, each of which can see every color in $[j]$. We will show that there are at most $(n-2k)k+k$ color degrees in $H$ that do not come from a color in $[j]$ by contradiction. Suppose that there are $(n-2k)k+k+1$ edges from $H$ to $M$. By the pigeon hole principle, there exists an edge $u_iv_i\in M$ that receives at least $n-2k+2$ edges from $H$. Notice that each vertex in $H$ can send at most two edges to $u_iv_i$. Therefore, there must exist two vertices in $H$ that each send two edges to $u_iv_i$, witnessing a $C_4$ subgraph; this is a contradiction. 

Now consider

\begin{align*}
    2mn&\leq \sum_{i=1}^k\hat d(u_i)+\hat d(v_i)+\sum_{v\in H} \hat d_G(v)\\
    &\leq \sum_{i=1}^j\hat d(u_i)+d(v_i)+\sum_{i=j+1}^k\hat d(u_i)+\hat d(v_i)+\sum_{v\in H} \left( \hat d_H(v)+k\right)\\
    &\leq j(2k+2j)+(k-j)(n+1)+ (n-2k)(j+k)+k\\
    &= 2kj+2j^2+nk+k-nj-j+ nj+nk-2kj-2k^2+k\\
    &\leq 2j^2+2nk-j+2k-2k^2\\
    &\leq 2j^2-2k^2+2k-j-2n+2mn\\
    &< 2mn.
\end{align*}

This is a contradiction; therefore, $k\geq m$.
\end{proof}

\section{Properly Edge Colored Graphs}\label{proper}

In this section, we consider properly edge colored graphs. The idea to analyze a greedy algorithm that constructs a matching appears in \cite{DFLMPW} and \cite{KPY}. The algorithm employed in this section is similar, with some adjustments to take into account the weaker degree assumption.

\begin{thm}\label{properthm}
Let $c$ be a proper edge coloring of $G$ with $n\geq 8m$ and   $\hat d(G)\geq 2mn$. Then $c$ admits a rainbow matching of size $m$.
\end{thm}

\begin{proof} Assume that $G$ is an edge minimal counter example to Theorem \ref{properthm}. Consider the following algorithm: 

\begin{enumerate}
    \item set $G_0:=G$
    \item if there exists $v\in V(G_i)$ with $\hat d(v)\geq 3(m-i)+1$, then $G_{i+1}=G_i-v$ and return to 2
    \item else, if there exists color class $R$ with $|R|\geq 2(m-i)+1$, then $G_{i+1}=G_i-R$ and return to 2
    \item else, if there exists $uv\in E(G_i)$, then $G_{i+1}=G_i-u-v-c(uv)$ and return to 2
    \item return $i$
\end{enumerate}

\begin{claim}
Suppose the algorithm returns $k\leq m$. Then $G_i$ contains a matching of size $k-i$ for $0\leq i \leq k$
\end{claim}

We will prove the claim by reverse induction on $i$. If $i=k$, then $G_i$ is empty, and the claim is true. Assume that the claim is true for $i$. We will prove the claim for $i-1$. By the induction hypothesis, there exists a matching $M\subseteq G_i$ of size $k-i$. There are three cases:

\textbf{Case 1:} Assume $G_{i}=G_{i-1}-v$ where $\hat d(v)\geq 3(m-i)+1$. By construction, $v\notin V(M)$. Since $\hat d(v)\geq 3(m-i)+1$, there exists $u\in N(v)$, such that $u\notin V(M)$ and $c(uv)\notin c(M)$. Then $M'=M\cup \{uv\}$ is a rainbow matching of size $k-i+1$.

\textbf{Case 2:} Assume $G_{i}= G_{i-1}-R$ for some color $R$ with $|R|\geq 2(m-i)+1$. This implies that $c(e)\neq R$ for all $e\in E(M)$. Since $c$ is a proper coloring and $|R|\geq 2(m-i)+1$, there exist $e\in G_{i-1}$ such that $c(e)=R$ and $M'=M\cup \{e\}$ is a rainbow matching.

\textbf{Case 3:} Assume that $G_i= G_{i-1}-v-u-c(uv)$ for some $uv\in E(G_{i-1})$. By construction $N[u]\cup N[v]$ is disjoint from $V(M)$ and $c(e)\neq c(uv)$ for all $e\in M$. Therefore, $M'=M\cup \{uv\}$ is a rainbow matching.

This concludes the proof of the claim. Since $G$ is an edge minimal counter example, the algorithm applied to $G$ will return $k<m$. We will now derive a contradiction. 

Let $W(G_i)$ denote the difference of total color degree between $G_i$ and $G_{i-1}$ under $c$.

\begin{claim}
For all $1\leq i \leq k$, we have $W(G_i)\leq 2n$.
\end{claim}

\textbf{Case 1:} Assume $G_{i}=G_{i-1}-v$ where $\hat d(v)\geq 3(m-i)+1$. Notice that $v$ is incident to at most $n-1$ edges. Therefore, deleting $v$ will remove at most $2(n-1)$ color degrees.

\textbf{Case 2:} Assume $G_{i}= G_{i-1}-R$ for some color $R$ with $|R|\geq 2(m-i)+1$. Because $c$ is proper, $|R|\leq \lfloor n/2\rfloor$. Deleting all edges of color $R$ reduces the color degree by at most $n$. 

\textbf{Case 3:} Assume that $G_i= G_{i-1}-v-u-c(uv)$ for some $uv\in E(G_{i-1})$. Since $G_i$ is not constructed by step 2, we know that $\hat d(u),\hat d(v)\leq 3(m-i)$. Furthermore, since $G_i$ is not constructed by step 3, we know that $|c(uv)|\leq 2(m-i)$. This implies that 

\begin{align*}
    W(G_i)&= 2(\hat d(v)+\hat d(u))+2|c(uv)|\\
    &\leq 16(m-i)\\
    &\leq 2n.
\end{align*}

This concludes the proof of the claim.  Now we have

\[ 2nm\leq \hat d(G)=\sum_{i=1}^k W(G_i)\leq 2nk,\] which is a contradiction since $k<m$. Therefore, the theorem is proven.

\end{proof}

\section{General Edge-Colored Graphs}\label{general}

Theorem \ref{genthm} provides contrast for Theorems \ref{trifreethm}, \ref{c4freethm}, and \ref{properthm}. The proof of Theorem \ref{genthm} is similar to the proof of Theorem \ref{properthm}. However, the greedy algorithm has been modified to accommodate graphs that are not properly colored.

\begin{thm}\label{genthm}
Let $c$ be an edge coloring of $G$ be a graph with $\hat d(G)\geq 2mn$ and $n\geq 3m^2+4m$. Then $c$ admits a rainbow matching of size $m$. 
\end{thm}

\begin{proof}
 Assume that $G$ is an edge minimal counter example to Theorem \ref{genthm}. Since $G$ is edge  minimal, no color class can induce a $P_4$ (path on $4$ vertices) or a triangle. This follows from the fact that if a color class $R$ induces a $P_4$ or triangle, then an edge can be deleted without reducing the total color degree of the graph. Therefore, each color class in $G$ induces a forest of stars. Let $s(R)$ denote the number of components induced by the color class $R$. Consider the following algorithm: 

\begin{enumerate}
    \item set $G_0:=G$
    \item if there exists $v\in V(G_i)$ with $\hat d(v)\geq 3(m-i)+1$, then $G_{i+1}=G_i-v$ and return to 2
    \item else, if there exists color $R$ with $s(R)\geq 2(m-i)+1$, then $G_{i+1}=G_i-R$ and return to 2 
     \item else, if there exists a vertex $v$ and a color $R$ such that $\hat d^R(v)\geq 3(m-i)+1$, then $G_{i+1}=G_i-v-R$ and return to 2
    \item else, if there exists $uv\in E(G_i)$, then $G_{i+1}=G_i-u-v-c(uv)$ and return to 2
    \item return $i$
\end{enumerate}

Since this algorithm is so similar to the algorithm featured in the proof of Theorem \ref{properthm}, the only things that remain to be checked are that step $4$ lets us extend a matching, and that the bounds on steps 4 and 5 are still good.

Assume that $G_{i}= G_{i-1}-v-R$ where $\hat d^R(v)\geq 3(m-i)+1$. Let $M$ be a rainbow matching of size $k-i$ contained in $G_{i}$. Since $v\notin V(G_{i})$, $v\notin V(M)$. Furthermore, $M$ does not contain an edge with color $R$. Since $\hat d^R(v)\geq 2(m-i)+1$, there exists an edge $uv$ with $c(uv)=R$ and $u\notin M$. Then $M\cup\{uv\}$ is a rainbow matching of size $k-i+1$ contained in $G_{i-1}$. 

If $G_{i}= G_{i-1}-v-R$ where $\hat d^R(v)\geq 3(m-i)+1$, then 2 and 3 must have been rejected. The color $R$ contributes at most $n-3(m-i)$ color using edges that are not incident upon $v$. Since $\hat d(v)\leq 3(m-i)$ and $d(v)\leq n$, it follows that $W(G_i)\leq n-3(m-i)+\hat d(v)+d(v)\leq  n-3(m-i)+3(m-i)+n=2n$. 

Suppose $G_i=G_{i-1}-v-u-c(uv)$. Then steps 2, 3, and 4 must have been rejected. This implies that $\hat d(v),\hat d(u)\leq 3(m-i)$. Furthermore, each color at $v,u$ can be represented at most $3(m-i)$ times. Finally, the edges of color $c(uv)$ can induce at most $2(m-i)$ stars with $3(m-i)$ edges each. Therefore, deleting all $c(uv)$ colored edges reduces the color degree by at most $6m^2+2m$. Thus, $W(G_i)\leq 6m^2+8m\leq 2n$.

 Suppose that the algorithm terminates in $k<m$ steps. Now we have

\[ 2nm\leq \hat d(G)=\sum_{i=1}^k W(G_i)\leq 2nk,\] which is a contradiction since $k<m$. Therefore, the theorem is proven.

\end{proof}

\section{Future Work} 

Though we was not able to resolve Question \ref{avedegreeq} for all graphs, we think the answer is affirmative:

\begin{conj}
All edge colored graphs $G$ with $\hat d(G)\geq 2mn$ contain a rainbow matching of size $m$.
\end{conj}

It would also be interesting to know under which conditions there exists a matching of size $m+1$. It seems that a small improvement in the estimates in the proofs of Theorems \ref{trifreethm} and \ref{properthm} could yield this result for edge colored graphs $G$ with $\hat d(G)\geq 2mn$. In fact, it may be that the proper question to ask is whether any graph $G$ with $\hat d(G)\geq 2mn$ contains a rainbow matching of size $m+1$. 

\subsection*{Acknowledgements}
The author would like to thank Michael Ferrara for asking Question \ref{avedegreeq} and acknowledge the Graduate Research Workshop in Combinatorics and its participants for hosting a rainbow matching problem. The author extends gratitude to the referees for their careful peer reviewing and helpful comments.  Finally, the author is grateful for Michael Young's  feedback and mentorship while working on these problems.


\begin{thebibliography}{100}

\bibitem{DFLMPW} J. Diemunsch, M. Ferrara, A. Lo, C. Moffatt, F. Pfender, and P. S. Wenger, {Rainbow Matchings of Size $\delta(G)$ in Properly Edge-Colored Graphs}, \emph{The Electronic Journal of Combinaotrics} \textbf{19} (2012), no. 2, Paper 52 11pp.

\bibitem{GS} A. Gy\'{a}rf\'{a}s, G.N. S\'{a}rk\"{o}zy, Rainbow matchings and cycle-free partial transversals of Latin Squares, \emph{Discrete mathematics} \textbf{327} (2014. pp. 96--102.

\bibitem{KPY} A. Kostochka, F. Pfender, and M. Yancey, Large Rainbow Matchings in Large Graphs, arXiv:1204.3193 April 14, 2012.

\bibitem{KY} A. Kostochka, and M. Yancey, Large Rainbow Matchings in Edge-Colored Graphs, \emph{Combinatorics, Probability and Computing} \textbf{21} (2012). 255--263.


\bibitem{LSWW} T. D. LeSaulnier, C. Stocker, P. S. Wenger, and D. B. West, Rainbow Matching in Edge-Colored Graphs, \emph{The Electronic Journal of Combinatorics} \textbf{17} (2010). Paper N26, 5pp.

\bibitem{LW} H. Li, and G. Wang, Heterochromatic Matchings in Edge-Colored Graph, \emph{The Electronic Journal of Combinatorics} \textbf{15} (2008), Paper R138, 10pp.

\bibitem{LWZ} G. Liu, G. Wang, and J. Zhang, Existence of rainbow matchings in properly edge colored graphs, \emph{Frontiers of mathematics in China} \textbf{7}(3) (2012). pp 543--550.

\bibitem{LT} A. Lo, and T. S. Tan, A Note on Large Rainbow Matchings in Edge-colored Graphs, \emph{Graphs and Combinatorics} \textbf{30}(2) (2014), 389--393.

\bibitem{R} H. J. Ryser, Neuere Probleme der Kombinatorik, in ``Vortrage uber Kombinatorik Oberwolfach," Mathematisches Forschungsinstitut Oberwolfach, July 1967, 24--29.

\bibitem{S} S. K. Stein, Transversals of Latin squares and their generalizations, \emph{Pacific Journal of Mathematics} \textbf{59}
(1975), no. 2, 567--575.

\bibitem{W} G. Wang, Rainbow matchings in properly edge colored graphs, \emph{The Electronic Journal of Combinatorics} \textbf{18} (2011), Paper P162, 7pp. 







\end{thebibliography}
\end{document}